\documentclass{amsart}

\usepackage{graphicx}
\usepackage{bm}
\usepackage{color}
\usepackage{natbib}
\usepackage{amssymb, amsmath, amsthm}
\usepackage{bm}
\theoremstyle{plain}
\usepackage{graphicx}
\usepackage[aboveskip=-5pt,position=bottom]{caption}
\usepackage{url}
\usepackage{hyperref}

\newtheorem{lemma}{Lemma}[section]

\newtheorem{theorem}[lemma]{Theorem}
\newtheorem{cor}[lemma]{Corollary}
\newtheorem{prop}[lemma]{Proposition}
\newtheorem{exam}[lemma]{\normalfont \scshape
 Example}
\newtheorem{rem}[lemma]{\normalfont \scshape Remark}

\newcommand{\R}{\mathbb{R}}
\newcommand{\N}{\mathbb{N}}
\newcommand{\norm}[1]{\left\Vert#1\right\Vert}
\newcommand{\abs}[1]{\left\vert#1\right\vert}
\newcommand{\set}[1]{\left\{#1\right\}}

\newcommand{\bfx}{\bm{x}}
\newcommand{\bfzero}{\bm{0}}

\newcommand{\bfone}{\bm{1}}

\newcommand{\bfe}{\bm{e}}

\newcommand{\bft}{\bm{t}}
\newcommand{\bfU}{\bm{U}}
\newcommand{\bfu}{\bm{u}}

\newcommand{\bfZ}{\bm{Z}}
\newcommand{\bfz}{\bm{z}}

\newcommand{\bfeta}{\bm{\eta}}

\allowdisplaybreaks[4]
\begin{document}

\title[Idempotent $D$-Norms]{On Idempotent $\boldsymbol{D}$-Norms}%
\author{Michael Falk}
\address{University of W\"{u}rzburg, Institute of Mathematics,  Emil-Fischer-Str. 30, 97074 W\"{u}rzburg, Germany.}
\email{michael.falk@uni-wuerzburg.de}



\begin{abstract}
Replacing the spectral measure by  a random vector $\bfZ$ allows the representation of a max-stable distribution on $\R^d$ with standard negative margins via a norm, called \emph{$D$-norm}, whose generator is $\bfZ$.  The set of $D$-norms can be equipped with a commutative multiplication type operation, making it a semigroup with an identity element. This multiplication leads to idempotent $D$-norms. We characterize the set of idempotent $D$-norms.  Iterating the multiplication provides a track of $D$-norms, whose limit exists and is again a $D$-norm. If this iteration is repeatedly done on the same $D$-norm, then the limit of the track is idempotent.
\end{abstract}

\maketitle

\section{Introduction}\label{sec:Introduction}

A $D$-norm $\norm\cdot_D$ on $\R^d$ is defined via a random vector (rv) $\bfZ=(Z_1,\dots,Z_d)$ as follows. It is required that $Z_i\ge 0$ a.s. and  $E(Z_i)=1$, $1\le i\le d$. The $D$-norm corresponding to $\bfZ$ is then defined by
\[
\norm{\bfx}_D:=E\left(\max_{1\le i\le d}\left(\abs{x_i}Z_i\right)\right),\qquad \bfx\in\R^d,
\]
and $\bfZ$ is called \emph{generator} of $\norm\cdot_D$.

If we take for example $Z_i=1$, $1\le i\le d$, then we obtain
\[
\norm{\bfx}_D=\norm{\bfx}_\infty:=\max_{1\le i\le d}\abs{x_i}.
\]
If $\bfZ$ is a random permutation of the vector $(d,0,\dots,0)\in\R^d$ with equal probabilities, then we obtain the $L_1$-norm
\[
\norm{\bfx}_D=\norm{\bfx}_1:=\sum_{i=1}^d\abs{x_i},\qquad \bfx\in\R^d,
\]
 These are the two extreme cases of a $D$-norm and we obviously have
\begin{equation*}
\norm\cdot_\infty\le\norm\cdot_D\le\norm\cdot_1
\end{equation*}
for each $D$-norm $\norm\cdot_D$.

Let $X_1,\dots,X_d$ be independent and identically Fr\'{e}chet-distributed rv, i.e., $P(X_i\le x)=$ $\exp(-x^{-\lambda})$, $x>0$, with parameter $\lambda>1$. Denote by $\Gamma(\cdot)$ the usual Gamma function and note that $E(X_i)=\Gamma(1-\lambda^{-1})$. Then $\bfZ=(Z_1,\dots,Z_d)$ with
    \begin{equation*}\label{eq:frechet_generator}
    Z_i:=\frac{X_i}{\Gamma(1-\frac 1\lambda)},\quad i=1,\dots,d,
    \end{equation*}
    generates  the logistic norm, i.e.,
\[
E\left(\max_{1\le i\le d}(\abs{x_i}Z_i)\right)= \norm{\bfx}_{\lambda}=\left(\sum_{i=1}^d\abs{x_i}^\lambda\right)^{1/\lambda},\qquad \bfx\in\R^d.
\]

\begin{rem}\label{rem:generation_of_sms_rv}\upshape
The theory of $D$-norms is an offspring of multivariate extreme value theory, as we illustrate in what follows.

The rv $\bfeta=(\eta_1,\dots,\eta_d)$ is called \emph{standard max-stable} (sms) if each component follows the standard negative exponential distribution, i.e., $P(\eta_i\le x)=\exp(x)$, $x\le 0$, $1\le i\le d$, and if for each $n\in\N$
\[
P\left(n\max_{1\le i\le n} \bfeta^{(i)}\le \bfx\right) = P\left(\bfeta\le\frac{\bfx}n\right)^n = P\left(\bfeta\le\bfx\right),\qquad \bfx\le\bfzero\in\R^d,
\]
where $\bfeta^{(1)},\bfeta^{(2)},\dots$ are independent copies of $\bfeta$. All operations on vectors such as $\max$ or $\le$ are meant componentwise.

The distribution function (df) $G(\bfx):=P(\bfeta\le\bfx)$, $\bfx\in\R^d$, of a sms rv $\bfeta$ is called \emph{standard max-stable} as well. The well-known de Haan-Resnick-Pickands-Vatan representation of a sms df, see, e.g., \citet[Sections 4.2, 4.3]{fahure10}, can now be formulated in quite an elegant way via $D$-norms.

\begin{theorem}[Pickands, de Haan-Resnick, Vatan]
A function $G:(-\infty,0]^d\to[0,1]$ is a sms df $\iff$ there exists a \emph{$D$-norm} $\norm\cdot_D$ on $\R^d$ such that
\[
G(\bfx)=\exp\left(-\norm{\bfx}_D\right),\qquad \bfx\le\bfzero\in\R^d.
\]
\end{theorem}

Each sms rv $\bfeta$ can be generated in the following way. Consider a Poisson point process on $[0,\infty)$ with mean measure $r^{-2}\,dr$. Let $V_i$, $i\in\N$, be a realization of this point process, i.e., we can choose $V_i=1/\sum_{k=1}^i E_k$, where $E_1,E_2,\dots$ are independent and identically standard exponential distributed rv. Consider independent copies $\bfZ_1,\bfZ_2,\dots$ of a generator $\bfZ$ of the $D$-norm corresponding to $\bfeta$, which are also independent of the Poisson process. Then we have
\[
\bfeta=_D-\frac 1{\sup_{i\in\N} V_i\bfZ_i},
\]
which is a consequence of \citet[Lemma 9.4.7]{dehaf06} and elementary computations.
\end{rem}

The copula of an arbitrary sms df $G(\bfx)=\exp\left(-\norm{\bfx}_D\right)$, $\bfx\le\bfzero\in\R^d$, is given by
\[
C(\bfu)=G(\log(\bfu)) =\exp\left(-\norm{\log(\bfu)}_D\right),\qquad \bfu\in (0,1]^d.
\]
As each multivariate max-stable df can be obtained from a sms df by just transforming the margins (see, e.g., \citet[Lemma 5.6.8]{fahure10}), the copula of \emph{each} multivariate extreme value distribution is of the preceding form.

We have, moreover, by Taylor expansion of $\log(\cdot)$ and $\exp(\cdot)$ for $\bfx\ge\bfzero\in\R^d$
\[
\lim_{t\downarrow 0} \frac{1-C(1-t\bfx)} t = \lim_{t\downarrow 0} \frac{1- \exp\left(-\norm{\log(1-t\bfx)}_D\right)} t = \norm{\bfx}_D,
\]
and, thus, $\norm{\bfx}_D=:\lambda(\bfx)$ is the \emph{stable tail dependence function} introduced by \citet{huang92}.

The function
\[
D(\bft):= \norm{\left(t_1,\dots,t_{d-1},1-\sum_{i=1}^{d-1} t_i\right)}_D,
\]
defined on $\set{\bft\in[0,1]^{d-1}:\,\sum_{i=1}^{d-1}t_i\le 1}$ is known as \emph{Pickands dependence function}, and we have
\[
\norm{\bfx}_D =\norm{\bfx}_1 D\left(\frac{\abs{x_1}}{\norm{\bfx}_1},\dots, \frac{\abs{ x_{d-1}}}{\norm{\bfx}_1}\right),\qquad \bfx\in\R^d,
\]
which offers a different way to represent a sms df; see \citet[Section 4.3]{fahure10}.

\begin{rem}\upshape
The generator of a $D$-norm is not uniquely determined, even its distribution is not. Take again the $D$-norm $\norm\cdot_\infty$, which is generated by the constant rv $\bfZ=(1,\dots,1)$. But $\norm\cdot_\infty$ is generated by \emph{any} rv $(\xi,\dots,\xi)$, where $\xi\ge 0$ a.s. is a random variable with $E(\xi)=1$:
\[
E\left(\max_{1\le i\le d}\left(\abs{x_i}\xi\right)\right)=\left(\max_{1\le i\le d}\abs{x_i}\right) E(\xi)=\norm{\bfx}_\infty,\qquad \bfx\in\R^d.
\]

While the equation
\[
\norm{\bfx}_D =E\left(\max(\abs{x_1}Z_1,\dots,\abs{x_d}Z_d)\right),\qquad \bfx\in\R^d,
\]
does not uniquely determine the distribution of the generator $\bfZ$, the function
\[
\varphi(\bfx):= E\left(\max(1, x_1Z_1,\dots, x_dZ_d)\right),\qquad \bfx>\bfzero\in\R^d,
\]
actually does\footnote{Pointed out by Professor Chen Zhou during the workshop on \emph{Extreme Value Theory}, November 3-5, 2014, Besan\c{c}on, France}. This can easily be seen by using  $E\left(\max(1,x_1Z_1,\dots,x_dZ_d)\right)=\int_0^\infty 1- P(\max(1,x_1Z_1,\dots,x_dZ_d)\le t)\,dt$.

\citet[Theorem 3.1]{wansto10a} established the fact that the norms generated by $\bfZ^{(1)}$, $\bfZ^{(2)}$ coincide if and only if
\[
E\left(\abs{\sum_{i=1}^d x_i Z_i^{(1)}}\right) = E\left(\abs{\sum_{i=1}^d x_i Z_i^{(2)}}\right),\qquad \bfx\in\R^d.
\]
\citet{molss14} explored further implications of the above equivalence, called zonoid equivalence, linked to stochastic geometry.
\end{rem}

\begin{rem}\label{rem:characterization_of_generators}\upshape
Let $\norm\cdot$ be an arbitrary norm on $\R^d$. The proof of the de Haan-Resnick-Pickands-Vatan representation of a sms df (see, e.g. \citet[Sections 4.2, 4.3]{fahure10}), shows that for each $D$-norm there exists a generator $\bfZ=(Z_1,\dots,Z_d)$ with the additional property $\norm{\bfZ}=\mathrm{const}$ a.s. The distribution of this generator is uniquely determined.

If we choose in particular $\norm\cdot=\norm\cdot_1$, then $\norm{\bfZ}=\sum_{i=1}^d Z_i=\mathrm{const}$ a.s., which, together with $E\left(\sum_{i=1}^d Z_i\right)=d$ implies $\mathrm{const}=d$. As a consequence we, thus, obtain in particular that each $D$-norm has a generator $\bfZ$ with the additional property $\sum_{i=1}^d Z_i=d$. This will in particular be useful in the derivation of Proposition \ref{prop:idempotent_D-norms, bivariate_case}.

By considering only generators with the additional assumption that their components sum up to $d$, one can equip the set of $D$-norms on $\R^d$ with a Wasserstein-metric, such that it becomes a complete separable metric space, see \citet{aulfazo14}.
\end{rem}

\begin{rem}\upshape
The set of $D$-norms is closely related to the set of copulas. Let the rv $\bfU=(U_1,\dots,U_d)$ follow an arbitrary copula $C$ on $\R^d$, i.e., each component $U_i$ is on $(0,1)$ uniformly distributed. Then
\[
\bfZ:=2 \bfU
\]
is, obviously, the generator of a $D$-norm. Note, however, that not each $D$-norm can be generated this way. Take, for example, the bivariate independence $D$-norm $\norm{(x,y)}_1=\abs x+\abs y$ and suppose that there exists a rv $(U_1,U_2)$ following a copula such that
\[
\norm{(x,y)}_1 =2 E\left(\max\left(\abs x U_1,\abs y U_2\right)\right),\qquad (x,y)\in\R^2.
\]
Choose $x=y=1$. From the general equation
\begin{equation}\label{eqn:representation_of_max(a,b)}
\max(a,b)=\frac{a+b}2+\frac{\abs{a-b}}2,\qquad a,b\in\R,
\end{equation}
we obtain
\begin{align*}
2 &= 2 E\left(\frac{U_1+U_2} 2 + \frac{\abs{U_1-U_2}} 2\right) = 1+ E\left(\abs{U_1-U_2}\right)\\
&\iff E\left(\abs{U_1-U_2}\right) = 1\\
&\iff \abs{U_1-U_2}=1\quad \mathrm{a.s.}
\end{align*}
But as $U_1,U_2$ realize in $(0,1)$ a.s., we have $\abs{U_1-U_2}<1$ a.s. and, thus, a contradiction. The bivariate $D$-norm $\norm\cdot_1$, therefore, cannot be generated by $2(U_1,U_2)$. It is obvious that $\norm\cdot_\infty$ on $\R^d$ with $d\ge 3$ cannot be generated by $2\bfU$, as $\norm{\bfone}_1=d>2E\left(\norm{\bfU}_\infty\right)$.
\end{rem}

Based on the componentwise multiplication of their generators, we introduce in Section \ref{sec:multiplication_of_D-norms} a multiplication operation on the set of $D$-norms, which makes this set a commutative semigroup with an identity element. This leads to \emph{idempotent} $D$-norms, which are characterized in Section \ref{sec:idempotent_D-norms}. Iterating the multiplication provides a \emph{track} of $D$-norms. We will establish in Section \ref{sec:tracks_of_D-norms} the fact that the limit of a $D$-norm track is an idempotent $D$-norm, if the multiplication is repeatedly done with the same $D$-norm. An application to copulas is given in Corollary \ref{cor:limit_of_copula_track}.

The $D$-norm approach can be extended to functional extreme value theory, see \citet{aulfaho11}. In the present paper, however, we restrict ourself to the finite dimensional space.

\section{Multiplication of $D$-Norms}\label{sec:multiplication_of_D-norms}

Our approach towards sms df enables the following multiplication-type operation on $D$-norms. Choose two generators $\bfZ^{(1)}, \bfZ^{(2)}$ with corresponding $D$-norms $\norm\cdot_{D^{(1)}}$, $\norm\cdot_{D^{(2)}}$  and suppose that $\bfZ^{(1)}$, $\bfZ^{(2)}$ are independent. Then
\[
\bfZ:=\bfZ^{(1)}\bfZ^{(2)}
\]
is again a generator of a $D$-norm, which we denote by $\norm\cdot_{D^{(1)}D^{(2)}}$. Recall that all operations on vectors, such as the above multiplication, is meant componentwise. Clearly, the multiplication is commutative  $\norm\cdot_{D^{(1)}D^{(2)}}=\norm\cdot_{D^{(2)}D^{(1)}}$. The $D$-norm $\norm\cdot_{D^{(1)}D^{(2)}}$ does not depend on the particular choice of generators,  as follows from conditioning, see below.

\begin{rem}\label{rem:multiplication_of_generators_and_eta}\upshape Take two independent generators $\bfZ^{(1)}$, $\bfZ^{(2)}$ of two two $D$-norms $\norm{\cdot}_{D^{(1)}}$, $\norm{\cdot}_{D^{(2)}}$ on $\R^d$. Let $\bfZ_i^{(k)}$, $i\in\N$, be independent copies of $\bfZ^{(k)}$, $k=1,2$, being mutually independent as well. If $V_i$, $i\in\N$, are the points of an independent Poisson process with mean measure $r^{-2}\,dr$, then the rv
\[
\bfeta:=-\frac 1{\sup_{i\in\N}V_i\bfZ_i^{(1)}\bfZ_i^{(2)}}
\]
is sms with
\[
P(\bfeta\le\bfx)= \exp\left(-\norm{\bfx}_{D^{(1)}D^{(2)}}\right),\qquad \bfx\le\bfzero\in\R^d,
\]
see Remark \ref{rem:generation_of_sms_rv}.
\end{rem}

Let, for instance, $\bfZ^{(2)}$ be a generator of the $D$-norm $\norm\cdot_\infty$. Then we obtain by conditioning on $\bfZ^{(1)}$
\begin{align}\label{eqn:conditioning_technique}
\norm{\bfx}_{D^{(1)}D^{(2)}}&= E\left(\norm{\bfx\bfZ^{(1)}\bfZ^{(2)}}_\infty\right)\nonumber\\
&= \int E\left(\norm{\bfx\bfz^{(1)}\bfZ^{(2)}}_\infty\mid \bfZ^{(1)}=\bfz^{(1)}\right) \left(P*\bfZ^{(1)}\right)\left(d\bfz^{(1)}\right)\nonumber\\
&= \int E\left(\norm{\bfx\bfz^{(1)}\bfZ^{(2)}}_\infty\right) \left(P*\bfZ^{(1)}\right)\left(d\bfz^{(1)}\right)\nonumber\\
&= \int \norm{\bfx\bfz^{(1)}}_\infty \left(P*\bfZ^{(1)}\right)\left(d\bfz^{(1)}\right)\nonumber\\
&= E\left(\norm{\bfx\bfZ^{(1)}}_\infty\right)\nonumber\\
&= \norm{\bfx}_{D^{(1)}},\qquad \bfx\in\R^d,
\end{align}
i.e., $\norm{\cdot}_{D^{(1)}D^{(2)}}= \norm{\cdot}_{D^{(1)}}$.  The sup-norm $\norm\cdot_\infty$ is, therefore, the identity element within the set of $D$-norms, equipped with the above multiplication. There is, clearly, no other $D$-norm with this property.

Equipped with this commutative multiplication, the set of $D$-norms on $\R^d$ is, therefore, a semigroup with an identity element.

\begin{rem}\upshape 
When applied to the representation of an arbitrary sms rv $\bfeta$ in Remark \ref{rem:multiplication_of_generators_and_eta}, this implies that multiplication with an independent rv $\xi\ge 0$, $E(\xi)=1$, does not alter its distribution:
\[
\bfeta=_D -\frac 1{\sup_{i\in\N}V_i\bfZ_i} =_D -
\frac 1{\sup_{i\in\N}V_i\xi_i\bfZ_i},
\]
where $\xi_i$, $i\in\N$, are independent copies of $\xi$, also independent of $\bfZ_i$, $i\in\N$, and the Poisson process $\set{V_i:\,i\in\N}$.
\end{rem}

Take, on the other hand, as $\bfZ^{(2)}$ a generator of the $D$-norm $\norm\cdot_1$. Then we obtain
\begin{align*}
\norm{\bfx}_{D^{(1)}D^{(2)}}&= E\left(\norm{\bfx\bfZ^{(1)}\bfZ^{(2)}}_\infty\right)\\
&= \int E\left(\norm{\bfx\bfz^{(1)}\bfZ^{(2)}}_\infty\right) \left(P*\bfZ^{(1)}\right)\left(d\bfz^{(1)}\right)\\
&=\int \sum_{i=1}^d \abs{x_i} z_i^{(1)}\, \left(P*\bfZ^{(1)}\right)\left(d\bfz^{(1)}\right)\\
&= \sum_{i=1}^d \abs{x_i} E\left(Z_i^{(1)}\right)\\
&=  \sum_{i=1}^d \abs{x_i},\qquad \bfx\in\R^d,
\end{align*}
i.e., $\norm{\cdot}_{D^{(1)}D^{(2)}}= \norm{\cdot}_1$. Multiplication with the independence norm $\norm\cdot_1$ yields the independence norm and thus, $\norm\cdot_1$ can be viewed as the \emph{maximal attractor} among the set of $D$-norms. There is, clearly, no other $D$-norm with this property.

Applied to the representation of an arbitrary sms rv, this implies that
\[
-\frac 1{\sup_{i\in\N}V_i\bfZ_i \tilde\bfZ_i} =_D \bfeta,
\]
where $\bfeta$ is a sms rv with independent components, if $\tilde\bfZ_i$, $i\in\N$, are independent copies of a generator of $\norm\cdot_1$, also independent of $\bfZ_i$, $i\in\N$, and the Poisson process $\set{V_i:\,i\in\N}$.

\section{Idempotent $D$-Norms}\label{sec:idempotent_D-norms}

The maximum-norm $\norm\cdot_\infty$ and the $L_1$-norm $\norm\cdot_1$ both satisfy
\[
\norm\cdot_{D^2}:= \norm\cdot_{DD}= \norm\cdot_{D}.
\]
Such a $D$-norm will be called \emph{idempotent}. The problem suggests itself to characterize the set of idempotent $D$-norms. This will be achieved in the present section. It turns out that in the bivariate case $\norm\cdot_\infty$ and $\norm\cdot_1$ are the only idempotent $D$-norms, whereas in higher dimensions each idempotent $D$-norm is a certain combination of $\norm\cdot_\infty$ and $\norm\cdot_1$.

\begin{rem}\upshape
Speaking in terms of rv, we will characterize in this section the set of generators $\bfZ$ such that
\[
\bfeta =_D -\frac 1{\sup_{i\in\N} V_i\bfZ_i} =_D -\frac 1{\sup_{i\in\N} V_i\bfZ_i \tilde\bfZ_i},
\]
where $\bfZ_i$,  $\tilde\bfZ_i$, $i\in\N$, are independent copies of $\bfZ$, also independent of the Poisson process $\set{V_i:\,i\in\N}$ on $[0,\infty)$, with intensity measure $r^{-2}\, dr$, see Remark \ref{rem:multiplication_of_generators_and_eta}.
\end{rem}

The following auxiliary result will be crucial for the characterization of idempotent $D$-norms.

\begin{lemma}\label{lem:characterization_of_E(X+Y)=E(X)}
 Let $X$ be a rv with $E(X)=0$ and let $Y$ be an independent copy of $X$. If
 \[
 E(\abs{X+Y})=E(\abs X),
 \]
 then either $X=0$ or $X\in\set{-m,m}$ a.s. with $P(X=-m)=P(X=m)=1/2$ for some $m>0$. The reverse implication is true as well.
\end{lemma}

\begin{proof}
Suppose that $P(X=-m)=P(X=m)=1/2$ for some $m>0$. Then, obviously,
 \[
 E(\abs X)=m=E(\abs{X+Y}).
 \]
 Next we establish the reverse implication. Suppose that $X$ is not a.s the constant zero. Denote by $F$ the df of $X$. Without loss of generality we can assume the representation $X=F^{-1}(U_1)$, $Y=F^{-1}(U_2)$, where $U_1,U_2$ are independent, on $(0,1)$ uniformly distributed rv and $F^{-1}(q):=\inf\set{t\in\R:\,F(t)\ge q}$, $q\in(0,1)$, is the generalized inverse of $F$. The well known equivalence
 \[
 F^{-1}(q)\le t\iff q\le F(t),\qquad q\in(0,1),\,t\in\R,
 \]
 (see, e.g. \citet[equation (1.2.9)]{reiss89}) together with Fubini's theorem implies
 \begin{align*}
 &E(\abs{X+Y})\\
 &=E\left(\abs{F^{-1}(U_1)+F^{-1}(U_2)}\right)\\
 &=\int_0^1\int_0^1 \abs{F^{-1}(u)+F^{-1}(v)}\,du\,dv\\
 &=-\int_0^{F(0)} \int_0^{F(0)} F^{-1}(u) + F^{-1}(v)\,du\,dv + \int_{F(0)}^1 \int_{F(0)}^1 F^{-1}(u) + F^{-1}(v)\,du\,dv\\
 &\hspace*{1cm} + 2 \int_0^{F(0)} \int_{F(0)}^1 \abs{F^{-1}(u) + F^{-1}(v)}\,du\,dv\\
 &=- \int_0^{F(0)}\left(F(0)F^{-1}(v) + \int_0^{F(0)} F^{-1}(u)\,du\right)\,dv\\
 &\hspace*{1cm} +\int_{F(0)}^1\left(\left(1-F(0)\right) F^{-1}(v) + \int_{F(0)}^1 F^{-1}(u)\,du\right)\,dv\\
 &\hspace*{1cm} + 2 \int_0^{F(0)} \int_{F(0)}^1 \abs{F^{-1}(u) + F^{-1}(v)}\,du\,dv\\
 &=-2F(0) \int_0^{F(0)} F^{-1}(v)\,dv + 2(1-F(0)) \int_{F(0)}^1 F^{-1}(v)\,dv\\
 &\hspace*{1cm} + 2 \int_0^{F(0)} \int_{F(0)}^1 \abs{F^{-1}(u) + F^{-1}(v)}\,du\,dv
 \end{align*}
 and
 \[
 E(\abs X)=-\int_0^{F(0)}F^{-1}(u)\,du+ \int_{F(0)}^1 F^{-1}(u)\,du.
 \]
 From the assumption $E(\abs{X+Y})=E(\abs X)$ we, thus, obtain the equation
 \begin{align*}
 0&= (1-2F(0)) \int_0^{F(0)} F^{-1}(v)\,dv+ (1-2F(0)) \int_{F(0)}^1 F^{-1}(v)\,dv\\
 &\hspace*{1cm} + 2 \int_0^{F(0)} \int_{F(0)}^1 \abs{F^{-1}(u) + F^{-1}(v)}\,du\,dv\\
 \intertext{or}
 0&= (1-2F(0))\int_0^1 F^{-1}(v)\,dv + 2 \int_0^{F(0)} \int_{F(0)}^1 \abs{F^{-1}(u) + F^{-1}(v)}\,du\,dv.
 \end{align*}
 The assumption $0=E(X)=\int_0^1 F^{-1}(v)\,dv$ now yields
 \[
 \int_0^{F(0)} \int_{F(0)}^1 \abs{F^{-1}(u) + F^{-1}(v)}\,du\,dv = 0
 \]
 and, thus,
 \begin{equation}\label{eqn:crucial_equation_in_proof}
 F^{-1}(u)+F^{-1}(v)= 0\qquad \mathrm{\ for\ }\lambda\mathrm{-a.e.\ } (u,v)\in[0,F(0)]\times[F(0),1],
 \end{equation}
 where $\lambda$ denotes the Lebesgue-measure on $[0,1]$.

 If $F(0)=0$, then $P(X>0)=1$ and, thus, $E(X)>0$, which would be a contradiction. If $F(0)=1$, then $P(X<0)>0$ unless $P(X=0)=1$, which we have excluded, and, thus, $E(X)<0$, which would again be a contradiction. We, consequently, have established $0<F(0)<1$.

 As the function $F^{-1}(q)$, $q\in(0,1)$, is in general continuous from the left (see, e.g., \citet[Lemma A.1.2]{reiss89}), equation \eqref{eqn:crucial_equation_in_proof} implies that $F^{-1}(v)$ is a constant function on $(0,F(0)]$ and on $(F(0),1)$, precisely,
 \[
 F^{-1}(v) = \begin{cases}
 -m,&v\in (0,F(0)],\\
 m, &v\in (F(0),1),
 \end{cases}
 \]
 for some $m>0$. Note that the representation $X=F^{-1}(U_1)$ together with the assumption that  $X$ is not a.s. the constant zero, implies  $m\not=0$. The condition
 \[
 0=E(X)=\int_0^{F(0)}F^{-1}(v)\,dv + \int_{F(0)}^1 F^{-1}(v)\,dv = m(1-2F(0))
 \]
 implies $F(0)=1/2$ and, thus,
 \[
 X=F^{-1}(U_1) =\begin{cases}
 m,&U_1>\frac 12,\\
 -m,&U_1\le \frac12,
 \end{cases}
 \]
 which is the assertion.
\end{proof}

The next Proposition is the first main result of this section.

\begin{prop}\label{prop:idempotent_D-norms, bivariate_case}
A bivariate $D$-norm $\norm\cdot_D$ is idempotent $\Leftrightarrow$ $\norm\cdot_D\in\set{\norm\cdot_1,\norm\cdot_\infty}$.
\end{prop}

\begin{proof}
It suffices to establish the implication
\[
\norm\cdot_{D^2}=\norm\cdot_D,\,\norm\cdot_D\not=\norm\cdot_\infty\implies \norm\cdot_D=\norm\cdot_1.
\]

Let $\bfZ^{(1)}=\left(Z_1^{(1)}, Z_2^{(1)}\right)$, $\bfZ^{(2)}=\left(Z_1^{(2)}, Z_2^{(2)}\right)$ be independent and identically distributed generators of $\norm\cdot_D$. According to Remark \ref{rem:characterization_of_generators} we can assume that $Z_1^{(1)}+Z_2^{(1)}=2=Z_1^{(2)}+Z_2^{(2)}$. Put $X:=Z_1^{(1)}-1$, $Y:=Z_1^{(2)}-1$. Then $X,Y$ are independent and identically distributed with $X\in[-1,1]$, $E(X)=0$. From equation \eqref{eqn:representation_of_max(a,b)}
we obtain the representation
\begin{align*}
&E\left(\max\left(Z_1^{(1)}Z_1^{(2)}, Z_2^{(1)}Z_2^{(2)}\right)\right)\\
&=E\left(\frac{Z_1^{(1)}Z_1^{(2)}}2 + \frac{Z_2^{(1)}Z_2^{(2)}}2\right) + \frac 12 E\left(\abs{Z_1^{(1)}Z_1^{(2)} - Z_2^{(1)}Z_2^{(2)}}\right)\\
&= 1 + E\left(\abs{Z_1^{(1)}-1 + Z_1^{(2)}-1}\right)\\
&= 1+ E(\abs{X+Y})
\end{align*}
as well as
\[
E\left(\max\left(Z_1^{(1)},Z_2^{(2)}\right)\right) = 1+ E(\abs X).
\]
Lemma \ref{lem:characterization_of_E(X+Y)=E(X)} now implies that $P(X=m)=P(X=-m)=1/2$ for some $m\in(0,1]$. It remains to show that $m=1$.

Set $x=1$ and $y=a$, where $0<a<1$ satisfies  $a(1+m)>1-m$. Then $a(1+m)^2>(1-m)^2$ as well, and we obtain by equation \ref{eqn:representation_of_max(a,b)}
\begin{align*}
\norm{(x,y)}_{D^2} &= E\left(\max\left(Z_1^{(1)}Z_1^{(2)}, a\left(2-Z_1^{(1)}\right) \left(2-Z_1^{(2)}\right)\right)\right)\\
&=\frac 14 \max\left((1-m)^2, a(1+m)^2\right) + \frac14 \max\left((1+m)^2,a(1-m)^2\right)\\
&\hspace*{1cm} + \frac 12 \max\left(1-m^2,a(1-m^2)\right)\\
&= \frac 14 a(1+m)^2 + \frac 14(1+m)^2 +\frac 12 (1-m^2)\\
&= \frac 14 (1+m)^2(1+a) + \frac 12 (1-m^2)
\end{align*}
and
\begin{align*}
\norm{(x,y)}_D &= E\left(\max\left(Z_1^{(1)}, a\left(2-Z_1^{(1)}\right)\right)\right)\\
&=\frac 12 \max(1+m,a(1-m)) + \frac 12 \max(1-m,a(1+m))\\
&= \frac 12 (1+m) + \frac 12 a(1+m)\\
&= \frac 12 (1+m)(1+a).
\end{align*}
From the equality $\norm{(x,y)}_{D^2}=\norm{(x,y)}_D$ and the fact that $1+m>0$ we, thus, obtain
\begin{align*}
&\frac 14 (1+m)(1+a) + \frac 12 (1-m) = \frac 12 (1+a)\\
&\iff (m-1)(a-1)=0\\
&\iff m=1,
\end{align*}
which completes the proof.
\end{proof}

Next we will extend Proposition \ref{prop:idempotent_D-norms, bivariate_case} to arbitrary dimension $d\ge 2$. Denote by $\bfe_i:=(0,\dots,0,1,0,\dots,0)\in\R^d$ the $i$-th unit vector in $\R^d$, $1\le i\le d$, and let $\norm\cdot_D$ be an arbitrary $D$-norm on $\R^d$. Then
\[
\norm{(x,y)}_{D_{i,j}}:=\norm{x\bfe_i+y\bfe_j}_D,\qquad (x,y)\in\R^2,\;1\le i<j\le d,
\]
defines a $D$-norm on $\R^2$, called \emph{bivariate projection} of $\norm\cdot_D$. If $\bfZ=(Z_1,\dots,Z_d)$ is a generator of $\norm\cdot_D$, then $(Z_i,Z_j)$ generates $\norm\cdot_{D_{i,j}}$.

\begin{prop}\label{prop:idempotent_norms:general_case}
Let $\norm\cdot_D$ be a $D$-norm on $\R^d$ such that each bivariate projection $\norm\cdot_{D_{i,j}}$ is different from the bivariate sup-norm $\norm\cdot_\infty$. Then $\norm\cdot_D$ is idempotent $\Leftrightarrow$ $\norm\cdot_D=\norm\cdot_1$.
\end{prop}

\begin{proof}
If $\norm\cdot_D$ is idempotent, then each bivariate projection is an idempotent $D$-norm on $\R^2$ and, thus, each bivariate projection is by Proposition \ref{prop:idempotent_D-norms, bivariate_case} necessarily the bivariate $L_1$-norm $\norm\cdot_1$. This implies bivariate independence of the margins of the sms df $G(\bfx)=\exp\left(-\norm{\bfx}_D\right)$, $\bfx\le\bfzero\in\R^d$. It is well-known that bivariate independence of the margins of $G$ implies complete independence (see, e.g., \citet[Theorem 4.3.3]{fahure10}) and, thus, $\norm\cdot_D=\norm\cdot_1$ on $\R^d$.
\end{proof}

If we allow bivariate complete dependence, then we obtain the complete class of idempotent $D$-norms on $\R^d$ as mixtures of lower-dimensional $\norm\cdot_\infty$- and $\norm\cdot_1$-norms. To this end we will first introduce the complete dependence frame of a $D$-norm.

Let $D$ be an arbitrary $D$-norm on $\R^d$ such that at least one bivariate projection $\norm\cdot_{D_{i,j}}$ equals $\norm\cdot_\infty$ on $\R^2$. Then there exist nonempty disjoint subsets $A_1,\dots,A_K$ of $\set{1,\dots,d}$, $1\le K< d$, $\abs{A_k}\ge 2$, $1\le k\le K$, such that
\[
\norm{\sum_{i\in A_k} x_i\bfe_i}_D =\max_{i\in A_k} \abs{x_i},\qquad \bfx\in\R^d,\,1\le k\le K,
\]
and no other projection $\norm{\sum_{i\in B} x_i\bfe_i}_D$, $B\subset\set{1,\dots,d}$, $\abs B\ge 2$, $B\not= A_k$, $1\le k\le K$, is the sup-norm $\norm\cdot_\infty$ on $\R^{\abs B}$. We call $A_1,\dots,A_K$ the \emph{complete dependence frame} (CDF) of $\norm\cdot_D$. If there is no completely dependent bivariate projection of $\norm\cdot_D$, then we say that its CDF is \emph{empty}.

To illustrate the significance of $A_1,\dots,A_K$, take a sms rv $\bfeta=(\eta_1,\dots,\eta_d)$ with df $G(\bfx)=\exp\left(-\norm{\bfx}_D\right)$, $\bfx\le\bfzero\in\R^d$. Then the sets $A_1,\dots,A_K$ assemble the indices of completely dependent components $\eta_i=\eta_j$ a.s., $i,j\in A_k$, and the sets $A_k$ are maximally chosen, i.e., we \emph{do not} have $\eta_i=\eta_j$ a.s. if $i\in A_k$ for some $j\in A_k^\complement$.

The next result characterizes the set of idempotent $D$-norms with at least one completely dependent bivariate projections.

\begin{theorem}
Let $\norm\cdot_D$ be an idempotent $D$-norm with non empty CDF $A_1,\dots,A_K$. Then we have
\[
\norm{\bfx}_D=\sum_{k=1}^K\max_{i\in A_k}\abs{x_i} + \sum_{i\in\set{1,\dots,d}\backslash \cup_{k=1}^d A_k} \abs{x_i},\qquad \bfx\in\R^d.
\]
On the other hand, the above equation defines for each set of nonempty disjoint subsets $A_1,\dots,A_K$ of $\set{1,\dots,d}$ with $\abs{A_k}\ge 2$, $1\le k\le K < d$, an idempotent $D$-norm on $\R^d$ with CDF $A_1,\dots,A_K$.
\end{theorem}

\begin{proof}
Let $\bfeta=(\eta_1,\dots,\eta_d)$ be a sms rv with df $G(\bfx)=\exp\left(-\norm{\bfx}_D\right)$, $\bfx\le\bfzero\in\R^d$. Then we have for $\bfx\le\bfzero\in\R^d$
\begin{align*}
G(\bfx)&= \exp\left(-\norm{\bfx}_D\right)\\
&=P(\eta_i\le x_i,\,1\le i\le d)\\
&= P\left(\eta_{k^*}\le \min_{i\in A_k} x_i,\,1\le k\le K;\,\eta_j\le x_j,\,j\in\left(\cup_{k=1}^K A_k\right)^\complement\right),
\end{align*}
where $k^*\in A_k$ is for each $k\in\set{1,\dots,K}$ an arbitrary but fixed element of $A_k$. The rv $\bfeta^*$ with joint components $\eta_{k^*}$, $1\le k\le K$, and $\eta_j$, $j\in\left(\cup_{k=1}^K A_k\right)^\complement$, is a sms rv of dimension less than $d$, and $\bfeta^*$ has no pair of completely dependent components. The rv $\bfeta^*$ might be viewed as the rv $\bfeta$ after the completely dependent components have been removed. Its corresponding $D$-norm is, of course, still idempotent. From Proposition \ref{prop:idempotent_norms:general_case} we obtain its df, i.e.,
\begin{align*}
G(\bfx)&=\exp\left(-\sum_{k=1}^K\abs{\min_{i\in A_k} x_i} - \sum_{j\in \left(\cup_{k=1}^K A_k\right)^\complement} \abs{x_j}\right)\\
&=\exp\left(-\sum_{k=1}^K\max_{i\in A_k} \abs{x_i} - \sum_{j\in \left(\cup_{k=1}^K A_k\right)^\complement} \abs{x_j}\right),\qquad \bfx\le\bfzero\in\R^d,
\end{align*}
which is the first part of the assertion.

Take, on the other hand, a rv $U$ that is  on the set of integers $\set{k^*:\,1\le k\le K}\cup \left(\cup_{k=1}^K A_k\right)^\complement$ uniformly distributed. Put $m:= K + \abs{\left(\cup_{k=1}^K A_k\right)^\complement}$ and set for $i=1,\dots,d$
\[
Z_i:=\begin{cases}
m,&i\in A_k,\\
0&\mathrm{otherwise},
\end{cases}
\]
if $U=k^*$, $1\le k\le K$, and
\[
Z_i:=\begin{cases}
m,&i=j,\\
0&\mathrm{otherwise},
\end{cases}
\]
if $U=j\in\left(\cup_{k=1}^K A_k\right)^\complement$. Then $E(Z_i)=1$, $1\le i\le d$, and
\begin{align*}
&E\left(\max_{1\le i\le d}\left(\abs{x_i}Z_i\right)\right)\\
&= \sum_{j\in \set{k^*:\,1\le k\le K}\cup \left(\cup_{k=1}^K A_k\right)^\complement} E\left(\max_{1\le i\le d}\left(\abs{x_i}Z_i\right)1(U=j)\right)\\
&=\sum_{k=1}^K \max_{i\in A_k}\abs{x_i} + \sum_{j\in \left(\cup_{k=1}^K A_k\right)^\complement} \abs{x_j},\qquad \bfx\in\R^d.
\end{align*}
It is easy to see that this $D$-norm is idempotent, which completes the proof.
\end{proof}

The set of all idempotent trivariate $D$-norms is, for example, given by
\[
\norm{(x,y,z)}_D =\begin{cases}
\max(\abs x, \abs y, \abs z)\\
\max(\abs x, \abs y) + \abs z\\
\max(\abs x, \abs z) + \abs y\\
\max(\abs y, \abs z) + \abs x\\
\abs x + \abs y + \abs z
\end{cases},
\]
where the three mixed versions are just permutations of the arguments and might be viewed as equivalent.

\section{Tracks of $D$-Norms}\label{sec:tracks_of_D-norms}

The multiplication of $D$-norms $D^{(1)}, D^{(2)},\dots$ on $\R^d$ can obviously be iterated:
\[
\norm\cdot_{\prod_{i=1}^{n+1}D^{(i)}} := \norm\cdot_{D^{(n+1)}\prod_{i=1}^n D^{(i)}}, \qquad n\in\N.
\]
This operation is commutative as well. In this section we investigate such \emph{$D$-norm tracks} $\norm\cdot_{\prod_{i=1}^n D^{(i)}}$, $n\in\N$. We will in particular show that each track converges to an idempotent $D$-norm if $\norm\cdot_{D^{(i)}}=\norm\cdot_D$, $i\in\N$, for an arbitrary $D$-norm $D$ on $\R^d$.

We start by establishing several auxiliary results. The first one indicates in particular that multiplication of $D$-norms decreases the dependence among the components of the corresponding sms rv.

\begin{lemma}\label{lem:multiplication_of_D-norms_is_monotone}
We have for arbitrary $D$-norms $\norm\cdot_{D^{(1)}}$, $\norm\cdot_{D^{(2)}}$ on $\R^d$
\[
\norm\cdot_{D^{(1)}D^{(2)}} \ge \max\left(\norm\cdot_{D^{(1)}}, \norm\cdot_{D^{(2)}}\right).
\]
\end{lemma}

\begin{proof}
 Let $\bfZ^{(1)}$, $\bfZ^{(2)}$ be independent generators of $\norm\cdot_{D^{(1)}}$, $\norm\cdot_{D^{(2)}}$.  We have for $\bfx\in\R^d$ by conditioning on $\bfZ^{(2)}$ as in equation \eqref{eqn:conditioning_technique}
 \begin{equation}\label{eqn:first_step_for_Jensen's_inequality}
 \norm{\bfx}_{D^{(1)}D^{(2)}} =E\left(\norm{\bfx\bfZ^{(1)}\bfZ^{(2)}}_\infty\right) = E\left(\norm{\bfx\bfZ^{(2)}}_{D^{(1)}}\right).
 \end{equation}
 Note that
 \begin{equation}\label{eqn:second_step_for_Jensen's_inequality}
 \norm{\bfx}_{D^{(1)}} = \norm{\bfx E\left(\bfZ^{(2)}\right)}_{D^{(1)}} =  \norm{ E\left(\bfx \bfZ^{(2)}\right)}_{D^{(1)}}.
 \end{equation}

 Put
 \[
 T(\bfx):= \norm{\bfx}_{D^{(1)}},\qquad \bfx\in\R^d.
 \]
 Then $T$ is a convex function by the triangle inequality and the homogeneity satisfied by any norm. We, thus, obtain from Jensen's together with equations \eqref{eqn:first_step_for_Jensen's_inequality} and \eqref{eqn:second_step_for_Jensen's_inequality}
 \begin{align*}
\norm{\bfx}_{D^{(1)}D^{(2)}}&=  E\left(\norm{\bfx\bfZ^{(2)}}_{D^{(1)}}\right)\\
&= E\left(T\left(\bfx \bfZ^{(2)}\right)\right)\\
&\ge T\left(E\left(\bfx\bfZ^{(2)}\right)\right)\\
&= \norm{ E\left(\bfx \bfZ^{(2)}\right)}_{D^{(1)}}\\
&= \norm{\bfx}_{D^{(1)}}.
 \end{align*}
 Exchanging $\bfZ^{(1)}$ and $\bfZ^{(2)}$ completes the proof.
 \end{proof}

 \begin{prop}\label{prop:the_limit_of_a_track_exists}
 Let $\norm\cdot_{D^{(n)}}$, $n\in\N$, be a set of arbitrary $D$-norms on $\R^d$. Then the limit of the track
 \[
 \lim_{n\to\infty} \norm{\bfx}_{\prod_{i=1}^n D^{(i)}}=: f(x)
 \]
 exists for each $\bfx\in\R^d$ and is a D-norm, i.e., $f(\cdot)=\norm\cdot_D$.
 \end{prop}

 \begin{proof}
  From Lemma \ref{lem:multiplication_of_D-norms_is_monotone} we know that for each $\bfx\in\R^d$ and each $n\in\N$
  \[
  \norm{\bfx}_{\prod_{i=1}^n D^{(i)}} \le \norm{\bfx}_{\prod_{i=1}^{n+1} D^{(i)}}.
  \]
  As each $D$-norm is bounded by the $L_1$-norm, i.e., $\norm{\bfx}_{\prod_{i=1}^n D^{(i)}}\le \norm{\bfx}_1$, the sequence $\norm{\bfx}_{\prod_{i=1}^n D^{(i)}}$, $n\in\N$, is monotone increasing and bounded and, thus, the limit
  \[
 \lim_{n\to\infty} \norm{\bfx}_{\prod_{i=1}^n D^{(i)}}=: f(x)
 \]
 exists in $[0,\infty)$. The triangle inequality and the homogeneity of $f(\cdot)$ are obvious. The monotonicity of the sequence $ \lim_{n\to\infty} \norm{\bfx}_{\prod_{i=1}^n D^{(i)}}$ implies that $f(\bfx)=0$ $\iff$ $\bfx=\bfzero$ and, thus, $f(\cdot)$ is a norm on $\R^d$. The characterization of a $D$-norm as established by \citet{hofm09} (see \citet[Theorem 4.4.2]{fahure10}) implies that $f(\cdot)$ is a $D$-norm as well.
   \end{proof}

If we set $D^{(n)}$ for each $n\in\N$ equal to a fixed but arbitrary $D$-norm, then the limit  in Proposition \ref{prop:the_limit_of_a_track_exists} is an idempotent $D$-norm.

\begin{theorem}\label{thm:limit_of_track_is_idempotent}
Let $\norm\cdot_D$ be an arbitrary $D$-norm on $\R^d$. Then the limit
\[
 \lim_{n\to\infty} \norm{\bfx}_{\prod_{i=1}^n D^{(i)}}=:\norm{\bfx}_{D^*},\qquad \bfx\in\R^d,
\]
is an idempotent $D$-norm on $\R^d$.
\end{theorem}

\begin{proof}
We know from Poposition \ref{prop:the_limit_of_a_track_exists} that $\norm\cdot_{D^*}$ is a $D$-norm on $\R^d$. Let $\bfZ^*$ be a generator of this $D$-norm and let $\bfZ^{(1)},\bfZ^{(2)},\dots$ be independent copies of the generator $\bfZ$ of $\norm\cdot_D$, independent of $\bfZ^*$ as well. Then we have for each $\bfx\in\R^d$
\[
\norm{\bfx}_{D^n}=E\left(\norm{\bfx \prod_{i=1}^n \bfZ^{(i)}}_\infty\right) \uparrow_{n\to\infty} \norm{\bfx}_{D^*}
\]
by Lemma \ref{lem:multiplication_of_D-norms_is_monotone}, as well as for each $k\in\N$
\begin{align*}
&\norm{\bfx}_{D^n}\\
&=E\left(\norm{\bfx \prod_{i=1}^k \bfZ^{(i)}  \prod_{j=k+1}^n \bfZ^{(j)}}_\infty\right)\\
&= \int E\left(\norm{\bfx \prod_{i=1}^k \bfz^{(i)}  \prod_{j=k+1}^n \bfZ^{(j)}}_\infty\right)\left(P*\left(\bfZ^{(1)},\dots,\bfZ^{(k)}\right)\right) \left(d\left(\bfz^{(1)},\dots,\bfz^{(k)}\right)\right)\\
&\to_{n\to\infty} \int \norm{\bfx \prod_{i=1}^k \bfz^{(i)}}_{D^*} \left(P*\left(\bfZ^{(1)},\dots,\bfZ^{(k)}\right)\right) \left(d\left(\bfz^{(1)},\dots,\bfz^{(k)}\right)\right)\\
&= E\left(\norm{\bfx\bfZ^* \prod_{i=1}^k \bfZ^{(i)}}_\infty\right)
\end{align*}
by the monotone convergence theorem. We, thus, have
\[
\norm{\bfx}_{D^*} =  E\left(\norm{\bfx\bfZ^* \prod_{i=1}^k \bfZ^{(i)}}_\infty\right)
\]
for each $k\in\N$. By letting $k$ tend to infinity and repeating the above arguments we obtain
\[
\norm{\bfx}_{D^*} =  E\left(\norm{\bfx\bfZ^* \prod_{i=1}^k \bfZ^{(i)}}_\infty\right)\uparrow_{k\to\infty} E\left(\norm{\bfx\bfZ^*}_{D^*}\right)= \norm{\bfx}_{D^*D^*},
\]
which completes the proof.
\end{proof}

If the initial $D$-norm $\norm\cdot_D$ has no complete dependence structure among its margins, i.e., if its CDF is empty, then the limiting $D$-norm in Theorem \ref{thm:limit_of_track_is_idempotent} is the $L_1$-norm. Otherwise, the limit has the same CDF as $\norm\cdot_D$.

The limit of an arbitrary track $\norm\cdot_{\prod_{i=1}^n D^{(i)}}$, $n\in\N$, is not necessarily idempotent. Take, for example, an arbitrary and non idempotent $D$-norm $\norm\cdot_D^{(1)}$ and $\norm\cdot_D^{(i)}=\norm\cdot_\infty$, $i\ge 2$. But it is an open problem, whether the limit of a track is again idempotent if $\norm\cdot_{D^{(i)}}\not=\norm\cdot_\infty$ for infinitely many $i\in\N$.

The following corollary is a consequence of the preceding results. Recall that $2\bfU$ is the generator of a $D$-norm if the rv $\bfU$ follows a copula.

\begin{cor}\label{cor:limit_of_copula_track}
Let $\bfU^{(1)}, \bfU^{(2)},\dots$ be independent copies of the rv $\bfU$ that follows an arbitrary copula on $\R^d$. Suppose that no pair $U_i,U_j$, $i\not=j$, of the components of $\bfU=(U_1,\dots,U_d)$ satisfies $U_i=U_j$ a.s. Then
\[
\lim_{n\to\infty}2^n E\left(\max_{1\le j\le d}\left(\abs{x_j} \prod_{i=1}^n U_j^{(i)}\right)\right)= \sum_{i=1}^n\abs{x_i}, \qquad\bfx\in\R^d.
\]
With $\bfx=(1,\dots,1)$ we obtain
\[
\lim_{n\to\infty} 2^nE\left(\max_{1\le j\le d} \prod_{i=1}^n U_j^{(i)}\right)=d.
\]
\end{cor}

\section*{Acknowledgements}
The author is indebted to an anonymous reviewer for bringing to his attention the papers by \citet{wansto10a} and \citet{molss14} as well as for pointing out that the conditions in an earlier version of Lemma \ref{lem:characterization_of_E(X+Y)=E(X)} can be weakened considerably.


\end{document}